\documentclass [12pt]{amsart}
\usepackage{graphicx}
\usepackage{geometry}
\usepackage{amsthm}
\usepackage{amssymb}

\theoremstyle{plain}
\newtheorem{thm}{Theorem}[section]
\newtheorem{cor}[thm]{Corollary}

\newtheorem{prop}[thm]{Proposition}
\newtheorem{defn}[thm]{Definition}
\newtheorem{exa}[thm]{Example}

\begin{document}

\title [{{On Graded Radically Principal Ideals}}]{On Graded Radically Principal Ideals}

 \author[{{R. Abu-Dawwas }}]{\textit{Rashid Abu-Dawwas }}

\address
{\textit{Rashid Abu-Dawwas, Department of Mathematics, Yarmouk University, Irbid, Jordan.}}
\bigskip
{\email{\textit{rrashid@yu.edu.jo}}}


\date{}

\begin{abstract} Let $R$ be a commutative $G$-graded ring with a nonzero unity. In this article, we introduce the concept of graded radically principal ideals. A graded ideal $I$ of $R$ is said to be graded radically principal if $Grad(I)=Grad(\langle c\rangle)$ for some homogeneous $c\in R$, where $Grad(I)$ is the graded radical of $I$. The graded ring $R$ is said to be graded radically principal if every graded ideal of $R$ is graded radically principal. We study graded radically principal rings. We prove an analogue of the Cohen theorem, in the graded case, precisely, a graded ring is graded radically principal if and only if every graded prime ideal is graded radically principal. Finally we study the graded radically principal property for the polynomial ring $R[X]$.
\end{abstract}

 \maketitle

\section{Introduction}

Throughout this article, all rings are commutative with a nonzero unity $1$. Let $G$ be a group with identity $e$. Then a ring $R$ is said to be $G$-graded if $R=\displaystyle\bigoplus_{g\in G} R_{g}$ with $R_{g}R_{h}\subseteq R_{gh}$ for all $g, h\in G$, where $R_{g}$ is an additive subgroup of $R$ for all $g\in G$. The elements of $R_{g}$ are called homogeneous of degree $g$. If $x\in R$, then $x$ can be written uniquely as $\displaystyle\sum_{g\in G}x_{g}$, where $x_{g}$ is the component of $x$ in $R_{g}$. Also, we set $h(R)=\displaystyle\bigcup_{g\in G}R_{g}$. Moreover, it has been proved in \cite{Nastasescue} that $R_{e}$ is a subring of $R$ and $1\in R_{e}$. Let $I$ be an ideal of a graded ring $R$. Then $I$ is said to be a graded ideal if $I=\displaystyle\bigoplus_{g\in G}(I\cap R_{g})$, i.e., for $x\in I$, $x=\displaystyle\sum_{g\in G}x_{g}$, where $x_{g}\in I$ for all $g\in G$. An ideal of a graded ring need not be graded (see \cite{Nastasescue}). If $I$ is a graded ideal of $R$, then $R/I$ is a graded ring with $\left(R/I\right)_{g}=(R_{g}+I)/I$ for all $g\in G$.

Let $I$ be a proper graded ideal of $R$. Then the graded radical of $I$ is $Grad(I)=\left\{x=\displaystyle\sum_{g\in G}x_{g}\in R:\mbox{ for all }g\in G,\mbox{ there exists }n_{g}\in \mathbb{N}\mbox{ such that }x_{g}^{n_{g}}\in I\right\}$. Note that $Grad(I)$ is always a graded ideal of $R$ (see \cite{Refai Hailat}).

\begin{prop}(\cite{Farzalipour})\label{1} Let $R$ be a $G$-graded ring.

\begin{enumerate}

\item If $I$ and $J$ are graded ideals of $R$, then $I+J$, $IJ$ and $I\bigcap J$ are graded ideals of $R$.

\item If $a\in h(R)$, then $\langle a\rangle$ is a graded ideal of $R$.
\end{enumerate}
\end{prop}

The concept of graded principal ideals have been introduced by Ashby in \cite{Ashby}. A graded ideal $I$ of $R$ is said to be graded principal if $I=\langle c\rangle$ for some $c\in h(R)$. The graded ring $R$ is said to be graded principal if every graded ideal of $R$ is graded principal. In this article, we follow \cite{AQ} to introduce the concept of graded radically principal ideals. A graded ideal $I$ of $R$ is said to be graded radically principal if $Grad(I)=Grad(\langle c\rangle)$ for some $c\in h(R)$. The graded ring $R$ is said to be graded radically principal if every graded ideal of $R$ is graded radically principal. We study graded rings with this property. Clearly, every graded principal ring is graded radically principal, we prove that the converse is not true in general.

Graded prime ideals have been introduced and studied by Refai, Hailat and Obiedat in \cite{Refai Hailat}. A proper graded ideal $P$ of $R$ is said to be graded prime if whenever $x, y\in h(R)$ such that $xy\in P$, then either $x\in P$ or $y\in P$. We prove the Cohen-type theorem for graded radically principal property, that is, a graded ring $R$ is graded radically principal if and only if every graded prime ideal is graded radically principal.

Atani and Tekir in \cite{Atani Tekir} introduced the avoidance graded prime theorem, that is, if $P\subseteq P_{1}\bigcup...\bigcup P_{n}$, where $P$ and $P_{i}$'s are graded prime ideals, then $P\subseteq P_{i}$ for some $i$, this property does not hold for infinite set of graded ideals $P_{i}$'s. We characterize graded rings with avoidance property for infinite set of graded prime ideals. Finally we study the graded radically principal property for the polynomial ring $R[X]$.

\section{Graded Radically Principal Ideals}

In this section, we introduce and study graded radically principal ideals.

\begin{defn}Let $R$ be a graded ring and $I$ be a graded ideal of $R$. Then $I$ is said to be graded radically principal if $Grad(I)=Grad(\langle c\rangle)$ for some $c\in h(R)$. The graded ring $R$ is said to be graded radically principal if every graded ideal of $R$ is graded radically principal.
\end{defn}

Clearly, every graded principal ring is graded radically principal. However, the next example shows that the converse is not true in general.

\begin{exa}\label{Example 2.3}Let $K$ be a field, $S=K[X, Y]$ and $G=\mathbb{Z}$. Then $S$ is $G$-graded by $S_{n}=\displaystyle\bigoplus_{i+j=n, i, j\geq0}Kx^{i}y^{j}$ with $deg(x)=deg(y)=1$. Then $I=\left\langle X^{2}, XY, Y^{2}\right\rangle$ is a graded ideal of $S$. Suppose that $R=S/I=K[X, Y]/I=K[x, y]$ where $x=\overline{X}$ and $y=\overline{Y}$. Then $R$ is $G$-graded by $R_{n}=\left(S/I\right)_{n}=(S_{n}+I)/I$. Note that, if $P$ is a graded prime ideal of $R$, then $x^{2}=y^{2}=0\in P$, and then $x, y\in P$, and hence $\langle x, y\rangle\subseteq P$. Since $\langle x, y\rangle$ is a graded maximal ideal of $R$, $P=\langle x, y\rangle$. Thus, the only graded prime ideal of $R$ is $P=\langle x, y\rangle=Grad(\langle0\rangle)$. Now, let $I$ be a graded ideal of $R$. Then either $I=R=Grad(\langle1\rangle)$ or $I\subseteq P$, and in this case, $Grad(I)=P=Grad(\langle0\rangle)$. Hence, $R$ is graded radically principal. On the other hand, $R$ is not graded principal since the graded ideal $P=\langle x, y\rangle$ is not graded principal.
\end{exa}

\begin{prop}\label{Proposition 2.4} Let $R$ be a graded ring and $I$ be a graded radically principal ideal of $R$. Then there exists $c\in I\bigcap h(R)$ such that $I =Grad(\langle c\rangle)$.
\end{prop}

\begin{proof}Since $I$ is graded radically principal, there exists $a\in h(R)$ such that $Grad(I)=Grad(\langle a\rangle)$, and then $a\in \langle a\rangle\subseteq Grad(\langle a\rangle)=Grad(I)$, which implies that $a^{n}\in I$ for some positive integer $n$. Now, $a\in h(R)$, so $a\in R_{g}$ for some $g\in G$, and then $a^{n}=\underbrace{a.a...a}_{n-times}\in \underbrace{R_{g}R_{g}...R_{g}}_{n-times}\subseteq R_{g^{n}}\subseteq h(R)$. Thus, $c=a^{n}\in I\bigcap h(R)$ such that $I\subseteq Grad(I)=Grad(\langle a\rangle)=Grad(\langle a^{n}\rangle)=Grad(\langle c\rangle)\subseteq I$.
\end{proof}

\begin{prop}\label{Proposition 2.5}Let $R$ be a graded ring. Suppose that $I$ and $J$ are two graded radically principal ideals of $R$. Then $IJ$ and $I\bigcap J$ are graded radically principal ideals of $R$.
\end{prop}

\begin{proof}By Proposition \ref{1}, $IJ$ and $I\bigcap J$ are graded ideals of $R$. Also, by Proposition \ref{Proposition 2.4}, $Grad(I)=Grad(\langle x\rangle)$ and $Grad(J)=Grad(\langle y\rangle)$ for some $x\in I\bigcap h(R)$ and $y\in J\bigcap h(R)$. So, $xy\in h(R)$ such that $Grad(IJ)=Grad\left(I\bigcap J\right)=Grad(I)\bigcap Grad(J)=Grad(\langle x\rangle)\bigcap Grad(\langle y\rangle)=Grad(\langle xy\rangle)$. Hence, $IJ$ and $I\bigcap J$ are graded radically principal ideals of $R$.
\end{proof}

\begin{prop}\label{Proposition 2.6} Let $R$ be a graded radically principal ring. Then $R/I$ is a graded radically principal ring for every graded ideal $I$ of $R$.
\end{prop}

\begin{proof}Let $J/I$ be a graded ideal of $R/I$. Then $J$ is a graded ideal of $R$, and then $J=Grad(\langle c\rangle)$ for some $c\in h(R)$. So, $\overline{c}\in h\left(R/I\right)$ such that $Grad\left(J/I\right)=Grad(\langle \overline{c}\rangle)$. Hence, $R/I$ is a graded radically principal ring.
\end{proof}

The next theorem gives an analogue of the Cohen-type theorem for graded radically principal rings.

\begin{thm}\label{Theorem 2.7} Let $R$ be a graded ring. Then $R$ is graded radically principal if and only if every graded prime ideal of $R$ is graded radically principal.
\end{thm}

\begin{proof}Suppose that every graded prime ideal of $R$ is a graded radically principal ideal. Assume that $X$ is the set of all graded ideals of $R$ that are not graded radically principal. We show that $X=\emptyset$. Suppose that $X\neq\emptyset$. Let $I_{0}\subseteq I_{1}\subseteq...\subseteq I_{n}\subseteq...$ be an increasing chain in $X$. Suppose that $I=\displaystyle\bigcup_{i}I_{i}$. Then $I$ is a graded ideal of $R$. If $I$ is graded radically principal, then $Grad(I)=Grad(\langle c\rangle)$ for some $c\in h(R)$. Since $c\in Grad(I)$, $c^{k}\in I$ for some positive integer $k$, and then $c^{k}\in I_{j}$ for some $j$, which implies that $c\in Grad(I_{j})$, and hence $Grad(I_{j})=Grad(\langle c\rangle)$, which is a contradiction. Thus, $I\in E$ and clearly, $I_{i}\subseteq I$ for all $i$. By Zorn's lemma, $X$ has a maximal element, say $P$. We show that $P$ is graded prime. Let $x, y\in h(R)$ such that $xy\in P$. Suppose that $x, y\notin P$, and let $P_{1}=P+\langle x\rangle$ and $P_{2}=P+\langle y\rangle$. Since $P$ is maximal in $X$ with $P\subsetneqq P_{1}$ and $P\subsetneqq P_{2}$, $P_{1}$ and $P_{2}$ are graded radically principal ideals of $R$, and then by Proposition \ref{Proposition 2.5}, $P_{1}P_{2}$ is a graded radically principal ideal of $R$, but $Grad(P_{1}P_{2})=Grad(P)$ since $P_{1}P_{2}\subseteq P$ and $P^{2}\subseteq P_{1}P_{2}$. This gives a contradiction. Hence, $P$ is a graded prime ideal of $R$ which is not graded radically principal, that is a contradiction. Thus, $X=\emptyset$. The converse is clear.
\end{proof}

Let $S\subseteq h(R)$ be a multiplicative set. Then $S^{-1}R$ is a graded ring with $(S^{-1}R)_{g}=\left\{\frac{a}{s},a\in R_{h}, s\in S\cap R_{hg^{-1}}\right\} $.

\begin{cor}\label{Corollary 2.8} Let $R$ be a graded radically principal ring. If $S$ is a multiplicative subset of $h(R)$, then $S^{-1}R$ is a graded radically principal ring.
\end{cor}

\begin{proof}Let $P$ be a graded prime ideal of $S^{-1}R$. Then $P=S^{-1}K$ for some graded prime ideal of $R$. Since $R$ is graded radically principal, $K=Grad(\langle c\rangle)$ for some $c\in h(R)$, and since $Grad(P)=S^{-1}K=S^{-1}Grad(\langle c\rangle)=Grad(S^{-1}\langle c\rangle)$, we have that $P$ is a graded radically principal ideal of $S^{-1}R$. Hence, by Theorem \ref{Theorem 2.7}, $S^{-1}R$ is a graded radically principal ring.
\end{proof}

Let $R_{1}$ and $R_{2}$ be two $G$-graded rings. Then $R_{1}\times R_{2}$ is $G$-graded ring by $(R_{1}\times R_{2})_{g}=(R_{1})_{g}\times (R_{2})_{g}$ for all $g\in G$.

\begin{cor}\label{Corollary 2.10}Let $R_{1}$, $R_{2}$,..., $R_{n}$ be $G$-graded rings and $R=R_{1}\times R_{2}\times...\times R_{n}$. Then $R$ is graded radically principal if and only if $R_{i}$ is graded radically principal for all $1\leq i\leq n$.
\end{cor}

\begin{proof}Suppose that $R$ is graded radically principal. Then by Proposition \ref{Proposition 2.6}, $R_{i}=R/I_{i}$, where $I_{i}=R_{1}\times R_{2}\times...\times R_{i-1}\times\{0\}\times R_{i+1}\times...\times R_{n}$, is graded radically principal for all $1\leq i\leq n$. Conversely, let $P$ be a graded prime ideal of $R$. Then there exists a graded prime ideal $K_{j}$ of $R_{j}$ for some $j$ such that $P=R_{1}\times...\times K_{j}\times...\times R_{n}$. Since $R_{j}$ is graded radically principal, there exists $c_{j}\in K_{j}\bigcap h(R_{j})$ such that $K_{j}=Grad(\langle c_{j}\rangle)$, and then $c=(1,..., c_{j},..., 1)\in h(R)$ such that $Grad(P)=Grad(\langle c\rangle)$. So, $P$ is a graded radically principal ideal of $R$, and hence by Theorem \ref{Theorem 2.7}, $R$ is a graded radically principal ring.
\end{proof}

\begin{thm}\label{Theorem 2.11} Let $R$ be a graded ring. Then $R$ is a graded radically principal ring if and only if for every graded prime ideal $P$ of $R$, we have $P\nsubseteq\displaystyle\bigcup_{P\nsubseteq K}K$.
\end{thm}

\begin{proof}Suppose that $R$ is a graded radically principal ring. Let $P$ be a graded prime ideal of $R$. Then $P=Grad(\langle c\rangle)$ for some $c\in h(R)$. If $K$ is a graded prime ideal of $R$ such that $P\nsubseteq K$, then $c\notin K$, and then $c\notin \displaystyle\bigcup_{P\nsubseteq K}K$. Hence, $P\nsubseteq\displaystyle\bigcup_{P\nsubseteq K}K$. Conversely, let $P$ be a graded prime ideal of $R$. Since $P\nsubseteq\displaystyle\bigcup_{P\nsubseteq K}K$, there exists $c\in P$ such that $c\notin K$ whenever $P\nsubseteq K$, and then $c_{g}\notin K$ for some $g\in G$. Note that $c_{g}\in P$ as $P$ is a graded ideal, and clearly, $Grad(\langle c_{g}\rangle)\subseteq P$. If $K$ is a graded prime ideal of $R$ containing $c_{g}$, then $P\subseteq K$. Thus, $P\subseteq\displaystyle\bigcap_{c_{g}\in K}K=Grad(\langle c_{g}\rangle)$. Hence, $P=Grad(\langle c_{g}\rangle)$. By Theorem \ref{Theorem 2.7}, $R$ is graded radically principal.
\end{proof}

For a graded ring $R$, it is well known that if $P$ is a graded prime ideal of $R$ such that $P\subseteq\displaystyle\bigcup_{i\in \Delta}P_{i}$, where $P_{i}$'s are graded prime ideals of $R$ and $\Delta$ is finite, then $P\subseteq P_{i}$ for some $i\in \Delta$. This result does not hold for an infinite set $\Delta$. The following result characterizes graded rings with this property, and the proof is immediate from Theorem \ref{Theorem 2.11}.

\begin{cor}\label{Corollary 2.12} Let $R$ be a graded ring. Then $R$ is a graded radically principal ring if and only if $R$ has the graded avoidance property, that is, if $P\subseteq\displaystyle\bigcup_{i\in \Delta}P_{i}$, where $P$ and $P_{i}$'s are graded prime ideals of $R$, then $P\subseteq P_{i}$ for some $i\in \Delta$.
\end{cor}

\begin{cor}\label{Corollary 2.13} Let $R$ be a graded ring. If $R$ has finitely many graded prime ideals, then $R$ is a graded radically principal ring.
\end{cor}

\begin{proof}If $R$ has finitely many graded prime ideals, then the graded avoidance property holds, and then by Corollary \ref{Corollary 2.12}, it follows that $R$ is graded radically principal.
\end{proof}

A graded commutative ring $R$ with unity is said to be a graded integral domain if $R$ has no homogeneous zero divisors. A graded commutative ring $R$ with unity is said to be a graded field if every nonzero homogeneous element of $R$ is unit. The next example shows that a graded field need not be a field.

\begin{exa}Let $R$ be a field and suppose that $F=\left\{x+uy:x, y\in R, u^{2}=1\right\}$. If $G=\mathbb{Z}_{2}$, then $F$ is $G$-graded by $F_{0}=R$ and $F_{1}=uR$. Let $a\in h(F)$ such that $a\neq0$. If $a\in F_{0}$, then $a\in R$ and since $R$ is a field, we have $a$ is a unit element. Suppose that $a\in F_{1}$. Then $a=uy$ for some $y\in R$. Since $a\neq0$, we have $y\neq0$, and since $R$ is a field, we have $y$ is a unit element, that is $zy=1$ for some $z\in R$. Thus, $uz\in F_{1}$ such that $(uz)a=uz(uy)=u^{2}(zy)=1.1=1$, which implies that $a$ is a unit element. Hence, $F$ is a graded field. On the other hand, $F$ is not a field since $1+u\in F-\{0\}$ is not a unit element since $(1+u)(1-u)=0$.
\end{exa}

\begin{prop}\label{Proposition 4.1} Let $R$ be a graded integral domain. Then $R$ is a graded field if and only if $R[X]$ is a graded radically principal ring.
\end{prop}

\begin{proof}Suppose that $R$ is a graded field. Then $R[X]$ is a graded principal domain, which implies that $R$ is a graded radically principal ring. Conversely, let $a\in h(R)-\{0\}$. Suppose that $I$ is the graded ideal of $R[X]$ generated by $a$ and $X$. Since $R[X]$ is graded radically principal, $Grad(I)=Grad(\langle f\rangle)$ for some homogeneous $f\in I$. Since $a\in I$, $a\in Grad(\langle f\rangle)$, and then there exists a positive integer $n$ such that $a^{n}=gf$ for some $g\in R[X]$. Since $R$ is a graded integral domain, $0=deg(a^{n})=deg(g)+deg(f)$, which implies that $f$ is a nonzero constant. Also, since $X\in I$, $X\in Grad(\langle f\rangle)$, and then there exists a positive integer $m$ such that $X^{m}=hf$ for some $h\in R[X]$. Since $f$ is constant, $1=b_{m}f$, where $b_{m}$ is the coefficient of $X^{m}$ in $h$, which implies that $f$ is a unit element, and then $I=R[X]$, so there exist $s, t\in R[X]$ such that $1=sa+tX$, and then $1=s(0)a$, where $s(0)\in R\subseteq h(R[X])$. Hence, $a$ is a unit element. So, $R$ is a graded field.
\end{proof}

\begin{exa}Let $K$ be a field, $R=K[X]$ and $G=\mathbb{Z}$. Then $R$ is $G$-graded by $R_{j}=KX^{j}$ for $j\geq0$, and $R_{j}=0$ otherwise. Since $K$ is a field, $K$ is a graded field, and then $K[X]$ is graded radically principal by Proposition \ref{Proposition 4.1}.
\end{exa}

\begin{exa}Let $K$ be a field, $R=K[X]$ and $G=\mathbb{Z}_{3}$. Then $R$ is $G$-graded by $R_{0}=\langle 1, x^{3}, x^{6},...\rangle$, $R_{1}=\langle x, x^{4}, x^{7},...\rangle$ and $R_{2}=\langle x^{2}, x^{5}, x^{8},...\rangle$. By Proposition \ref{Proposition 4.1}, $K[X]$ is graded radically principal.
\end{exa}

First strongly graded rings have been introduced and studied in \cite{Refai}, a $G$-graded ring $R$ is said to be first strong if $1\in R_{g}R_{g^{-1}}$ for all $g\in supp(R,G)$, where $supp(R, G)=\left\{g\in G:R_{g}\neq\{0\}\right\}$. In fact, it has been proved that $R$ is first strongly $G$-graded if and only if $supp(R,G)$ is a subgroup of $G$ and $R_{g}R_{h}= R_{gh}$ for all $g,h\in supp(R,G)$. We introduce the following:

\begin{prop}\label{7}Every $G$-graded field is first strongly graded.
\end{prop}

\begin{proof}Let $R$ be a $G$-graded field. Suppose that $g\in supp(R, G)$. Then $R_{g}\neq\{0\}$, and then there exists $0\neq x\in R_{g}$. Since $R$ is a graded field, we conclude that there exists $y\in h(R)$ such that $xy=1$. Since $y\in h(R)$, $y\in R_{h}$ for some $h\in G$, and then $1=xy\in R_{g}R_{h}\subseteq R_{gh}$. So, $0\neq1\in R_{gh}\bigcap R_{e}$, which implies that $gh=e$, that is $h=g^{-1}$. Hence, $1=xy\in R_{g}R_{g^{-1}}$, and thus $R$ is first strongly graded.
\end{proof}

\begin{cor}Let $R$ be a graded integral domain. If $R[X]$ is a graded radically principal ring, then $R$ is first strongly graded.
\end{cor}

\begin{proof}Apply Proposition \ref{Proposition 4.1} and Proposition \ref{7}.
\end{proof}

\begin{thm}\label{Theorem 4.3}Let $R$ be a graded ring. If $R[X]$ is graded radically principal, then $R$ is graded radically principal and every graded prime ideal of $R$ is graded maximal.
\end{thm}

\begin{proof}Since $R[X]$ is graded radically principal, by Proposition \ref{Proposition 2.6}, $R[X]/\langle X\rangle$ is graded radically principal, and then $R$ is graded radically principal. Let $P$ be a graded prime ideal of $R$. Then since $\left(R/P\right)[X]=R[X]/P[X]$ is graded radically principal by Proposition \ref{Proposition 2.6}, and then $R/P$ is a graded field by Proposition \ref{Proposition 4.1}. Hence, $P$ is a graded maximal ideal of $R$.
\end{proof}

\end{document}